\RequirePackage{fix-cm}
\documentclass[smallextended]{svjour3}       
\smartqed  
\usepackage{graphicx}

\usepackage{amsmath,amsfonts,amssymb,amsthm}
\usepackage{mathtools}
\usepackage[justification=raggedright]{caption}
\usepackage{url}
\usepackage{hyperref}
\usepackage{xcolor}
\usepackage{cite}
\theoremstyle{plain}
\theoremstyle{definition}
\theoremstyle{remark}

\begin{document}

\title{Asymmetric expansion preserves hyperbolic convexity. 
}
\subtitle{}


\author{Dhruv Kohli \and
        Jeffrey M. Rabin
}


\institute{Department of Mathematics\\
            University of California, San Diego\\
            La Jolla, CA 92093\\
            \email{\{dhkohli, jrabin\}@ucsd.edu}
            }

\date{January 2020}

\maketitle

\begin{abstract}
In an earlier paper we showed that the radial expansion of a hyperbolic convex set in the Poincar\'e disk about any point inside it results in a hyperbolic convex set. In this work, we generalize this result by showing that the asymmetric expansion of a hyperbolic convex set about any point inside it also results in a hyperbolic convex set. 
\keywords{Hyperbolic convexity, Poincar\'e disk, asymmetric expansion.}
\end{abstract}


\section{Introduction}
A convex set in the Euclidean or hyperbolic plane is one that contains the geodesic segment joining any pair of its points.
Convexity, being defined in terms of geodesics, is preserved by isometries.
In the Euclidean case it is also preserved by dilations or similarities and indeed by affine transformations. 
In a previous paper \cite{us}, we defined a hyperbolic analog of dilation, called (symmetric) radial expansion, and proved that it preserves hyperbolic convexity. 
In this paper we will generalize this result to the case of asymmetric radial expansion.
The Euclidean version of this is expansion by unequal factors $k_1,k_2 \ge 1$ along the $x$ and $y$ axes respectively. 
Equivalently, it acts in polar coordinates by sending $(r,\theta)$ to $(r',\theta')$, where $r'^2=r^2(k_1 \cos^2 \theta + k_2 \sin^2 \theta)$ and $k_1 \tan \theta' = k_2 \tan \theta$.

Let $\mathbb{D} \subseteq \mathbb{R}^2$ denote the Poincar\'e disk of unit radius parameterized as
\begin{align}
    X(r,\theta)=\tanh(r/2)(\cos\theta,\sin\theta), r > 0
\end{align}
and denote the origin $(0,0)$ by $0$. The coefficients of the first fundamental form are given by
\begin{align}
    E = 1,\ \ F = 0\ \  \text{and} \ \ G = \sinh^2 r, 
\end{align}
and the metric is $ds^2=dr^2 + \sinh^2 r \ d\theta^2$ which has curvature $-1$. The hyperbolic distance of $X(r,\theta)$ in $\mathbb{D}$ from $0$ is given by
\begin{align}
d(X(r,\theta)) = r.
\end{align}
Thus $(r,\theta)$ are geodesic polar coordinates centered at $0$.
Given $u,v \in \mathbb{D}$ there is a unique hyperbolic geodesic segment, denoted $[u,v]$, joining these points. When we treat $[u,v]$ as an oriented curve it will always be directed from $u$ toward $v$. Hyperbolic geodesics are arcs of Euclidean circles orthogonal to the unit circle, including Euclidean lines through the origin. A Euclidean circle centered at $a$ with radius $r$ is orthogonal to the unit circle $\partial \mathbb{D}$ if and only if $\left\|a\right\|^2=1+r^2$. The map
\begin{align}
    \tau_{c}(x) = \frac{(1-2c\cdot x +\left\|x\right\|^2)c + (1-\left\|c\right\|^2)x}{1-2c\cdot x +\left\|c \right\|^2\left\|x\right\|^2} 
\end{align}
is the unique hyperbolic isometry (see \cite{ungar}) that swaps the origin and $c \in \mathbb{D}$ with no rotation at the origin; in fact, $(\tau_c)'(0)=(1-\left\|c\right\|^2)I$.
\begin{definition}
    \label{rhconvex}
    A set $C \subseteq \mathbb{D}$ is hyperbolic convex (h-convex) if, for every $u,v \in C$, $[u,v]$ lies in $C$. Obviously, $C$ is h-convex if and only if $\tau_{c}(C)$ is h-convex for each $c \in \mathbb{D}$.
\end{definition}

\medskip

Given $\theta \in [-\pi/2,\pi/2)$ there is a unique hyperbolic geodesic $\gamma_0(\theta)$ emanating from $0$ with tangent vector $(\cos\theta,\sin\theta)$ at $0$. For $x \in \gamma_0(\theta)$, its asymmetric dilated image is the unique point given by $\delta_{0,(k_1,k_2)}(x)=x'$ on the geodesic $\gamma_0(\tan^{-1}((k_2/k_1)\tan\theta))$ with $d(x') = (k_1^2\cos^2\theta + k_2^2\sin^2\theta)^{1/2}d(x)$. If $x = X(r,\theta)$ then $x'=X(r',\theta')$ where $r' = r(k_1^2\cos^2\theta + k_2^2\sin^2\theta)^{1/2}$ and $\theta' = \tan^{-1}((k_2/k_1)\tan\theta)$. Thus, for $\theta \in [-\pi/2,\pi/2)$
\begin{align}
    \delta_{0,(k_1,k_2)}(X(r,\theta)) = X(r(k_1^2\cos^2\theta + k_2^2\sin^2\theta)^{1/2},\tan^{-1}((k_2/k_1)\tan\theta)). \label{asym_dil}
\end{align}
Then the asymmetric hyperbolic dilation about $c$ in $\mathbb{D}$ is given by $\delta_{c,(k_1,k_2)} = \tau_c \circ \delta_{0,(k_1,k_2)} \circ \tau^{-1}_c$. We refer to the asymmetric hyperbolic dilation of a point as asymmetric expansion when $k_1, k_2 \geqslant 1$. Note that it is straightforward to extend Eq. (\ref{asym_dil}) to the more general case when $\theta \in [-\pi,\pi)$ but that does not affect our result. 

\medskip

The main result of this paper is that \textit{asymmetric expansion of a h-convex set $C \subseteq \mathbb{D}$ about a point $c \in C$ preserves h-convexity i.e. $\delta_{c,(k_1,k_2)}(C)$ is h-convex when $k_1,k_2 \geq 1$}.

\medskip

We will make use of standard facts about hyperbolic convex sets which can be found in \cite{alexander, bishop, stoker}. Let $S$ be a connected set with nonempty interior whose boundary is a piecewise smooth curve. Then $S$ is h-convex iff it has a local support line at each boundary point. Here a local support line at a boundary point $p$ is a geodesic segment through $p$ such that all points of $S$ in a neighborhood of $p$ lie on the same side of the segment. Further, if $S$ is h-convex then the geodesic curvature of its oriented boundary has the same sign at all smooth points. If the boundary is oriented so that $k_g$ is nonnegative (but not identically zero) then the interior of $S$ is to the left. 

\section{Main Result and Proof}
\label{sec:result}
\begin{lemma}
    \label{lem0}
    Let $p,q$ be two distinct points in $\mathbb{D}$ such that $p,q \neq 0$. Let $\gamma$ and $\gamma'$ be two regular curves from $p$ to $q$ which lie in the region bounded by rays from $0$ towards $p$ and $q$ and whose geodesic curvatures do not change sign. Then, $\gamma$ and $\gamma'$ lie on the opposite sides of $[p,q]$ if and only if their geodesic curvatures have opposite signs.
\end{lemma}
\begin{proof}
    Suppose $\gamma$ and $\gamma'$ lie on opposite sides of $[p,q]$. Let $S$ be the section of $\mathbb{D}$ subtended by rays from $0$ towards $p$ and $q$. Clearly, $S$ is h-convex. Denote by $H_1$ and $H_2$ the two sets into which $[p,q]$ divides $S$. Both $H_1$ and $H_2$ are also h-convex. The curve $\gamma$, whose geodesic curvature does not change sign, also divides $S$ into two sets, one of which is h-convex. Clearly the set containing $[p,q]$ is h-convex and is denoted by $S_{\gamma}$. Let $C$ be the region enclosed by $\gamma$ and $[p,q]$. So $C$ is the intersection of $S_{\gamma}$ with either $H_1$ or $H_2$ (see Figure (\ref{fig1})). Since the intersection of h-convex sets is h-convex, therefore $C$ is h-convex. Using the same argument, the region $C'$ enclosed by $\gamma'$ and $[p,q]$ is also h-convex. We now show that $C \cup C'$ is also h-convex. It is sufficient to show that when $x \in C$ and $x' \in C'$ then $[x,x']$ lies in $C \cup C'$. Note that $[x,x']$ must intersect $[p,q]$; let the point of intersection be $w$. So $w$ lies in both $C$ and $C'$. Since $C$ and $C'$ are h-convex so $[x,w] \in C$ and $[w,x'] \in C'$. Therefore, $[x,x'] \in C \cup C'$. Since $\gamma - \gamma'$ (concatenation of $\gamma$ and reversed $\gamma'$) encloses $C \cup C'$ therefore the geodesic curvature of the smooth pieces of $\gamma - \gamma'$ (which are $\gamma$ and $-\gamma'$) do not change sign. So, the geodesic curvature of $\gamma$ and that of $-\gamma'$ have same sign. Since the geodesic curvatures of $\gamma'$ and $-\gamma'$ have opposite signs by definition, therefore, geodesic curvatures of $\gamma$ and $\gamma'$ have opposite signs.
    
    \medskip
    
    Now suppose $\gamma$ and $\gamma'$ lie on the same side of $[p,q]$. We show that their geodesic curvatures must have the same sign. Take a curve $\gamma''$ joining $p$ to $q$ on the other side of $[p,q]$ such that $\gamma''$ lies in the region bounded by rays originating from $0$ towards $p$ and $q$, and the geodesic curvature of $\gamma''$ does not change sign. From the argument above we conclude that geodesic curvatures of $\gamma''$ and $\gamma$ have opposite signs and the geodesic curvatures of $\gamma''$ and $\gamma'$ have opposite signs. So, the geodesic curvatures of $\gamma$ and $\gamma'$ must have the same sign.
    
    \begin{figure*}[h]
    \centering
    \includegraphics[width=0.4\textwidth]{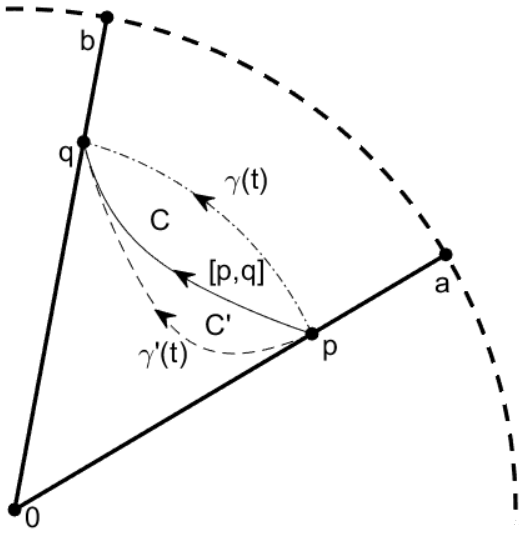}
    \caption{The sets $S$, $H_1$, $H_2$, $S_{\gamma}$, $S_{\gamma'}$ are represented by the regions enclosed by closed curves $[0a]\widetilde{ab}[b0]$, $[0p][pq][q0]$, $[pq][qb]\widetilde{ba}[ap]$, $[0p]\gamma[q0]$ and $[ap]\gamma'[qb]\widetilde{ba}$ respectively, where $\widetilde{ab}$ represent the dashed-circular arc from $a$ to $b$. The sets $C$ and $C'$ are shown in the figure.}
    \label{fig1}
    \end{figure*}
\end{proof}

We also make use of the following standard formula for geodesic curvature, which can be found in \cite{joprea} (in the case $v=1$) and in \cite{banchoff} (in terms of Christoffel symbols).

\begin{lemma}
    \label{lem01}
    Let $x(t)=X(r(t),\theta(t))$ be a regular curve in $\mathbb{D}$. Let $v(t)$ be its velocity at $x(t)$. Then the geodesic curvature at $x(t)$ is given by
    \begin{align}
        k_g &= \frac{1}{v^3}\sqrt{EG}\left(\frac{G_r}{G}r'^2\theta' +\frac{G_r}{2E}\theta'^3 +r'\theta'' -r''\theta' \right) \label{k_g_formula}
    \end{align}
    where $E(t) = 1$, $G(t) = \sinh^2r(t)$ and $G_r(t)=\sinh 2 r(t)$.
\end{lemma}

\begin{theorem}
    \label{thm1}
    Consider a h-convex set $C \subseteq \mathbb{D}$ and a point $c \in C$. Then for any $k_1,k_2\geqslant 1$, $\delta_{c,(k_1,k_2)}(C)$ is h-convex.
\end{theorem}
\begin{proof}
    Because of the invariance of h-convexity and asymmetric dilations under isometries of the hyperbolic metric, there is no loss of generality in assuming that $c = 0$. Also, there is no loss of generality in assuming that $k_2 = 1$ because
    \begin{align}
        \delta_{0,(k_1,k_2)}(X(r,\theta)) = \delta_{0,(k_1,1)}(\delta_{0,(1,k_2)}(X(r,\theta))), 
    \end{align}
    as can be checked by direct calculation.
    Finally, the case $k_1=k_2=1$ is trivial. So, we prove that for any $k > 1$, $\delta_{0,(k,1)}(C)$ is h-convex when $C$ is h-convex and $0\in C$. Denote $\delta_{0,(k,1)}(C)$ by $C'$.

    \medskip
    
    The following is an outline of the proof. We take two arbitrary points $\hat{x}_1,\hat{x}_2 \in C'$. We then parameterize $[\hat{x}_1,\hat{x}_2]$ using $\hat{x}(t)$ where $t \in [0,1]$, $\hat{x}(0) = \hat{x}_1$ and $\hat{x}(1)=\hat{x}_2$. To show that $C'$ is h-convex, we must show that $\hat{x}(t) \in C'$ for all $t \in [0,1]$. We compute the preimages of $\hat{x}_1, \hat{x}_2$ and $\hat{x}(t)$ under the map $\delta_{0,(k,1)}$. We call these preimages $x_1, x_2$ and $x(t)$. Since $\hat{x}_1, \hat{x}_2 \in C'$, therefore $x_1,x_2 \in C$. We must show that $x(t) \in C$ for all $t \in [0,1]$ which will then prove that $\hat{x}(t) = \delta_{0,(k,1)}(x(t)) \in C'$. To show that $x(t) \in C$, we first note that $x(t)$ lies in the region between the rays originating from $0$ towards $x_1$ and $x_2$ respectively. Then we show that $x(t)$ has negative geodesic curvature. We will construct a curve $\gamma(t)$ from $x_1$ to $x_2$ which lies on the side of $[x_1,x_2]$ opposite to $0$ and has positive geodesic curvature. Then using Lemma (\ref{lem0}) we conclude that $x(t)$ and $\gamma(t)$ lie on opposite sides of $[x_1,x_2]$ so that $x(t)$ lies on the side of $[x_1,x_2]$ where $0$ lies. Finally, using the facts that $x_1,x_2$ and $0$ lie in $C$ and $C$ is h-convex we conclude that $x(t) \in C$ and therefore $C'$ is h-convex.

    \medskip

    For conciseness we make use of the following functions when required
    \begin{align}
    \phi(a) &= \sinh a - a, \\
    \psi(a) &= a \coth a - 1.
    \end{align}

    \medskip 

    For convenience, denote $s = 1/k$, so $s \in (0,1)$. Consider two points $\hat{x}_1, \hat{x}_2 \in C'$ as follows,
    \begin{align}
    \hat{x}_1 &= X(\hat{r}_1,\hat{\theta}_1),\\
    \hat{x}_2 &= X(\hat{r}_2,\hat{\theta}_2),
    \end{align}
    where $\hat{r}_1, \hat{r}_2 > 0$. Without loss of generality, assume that $-\pi/2 \leqslant  \hat{\theta}_1 < \hat{\theta}_2 < \pi/2$. Let $\hat{x}(t)$ be the curve representing $[\hat{x}_1,\hat{x}_2]$ where $t\in [0,1]$. Then
    \begin{align}
    \hat{x}(t) &= X(\hat{r}(t),\hat{\theta}(t)), 
    \end{align}
    where $\hat{\theta}(t) = (1-t)\hat{\theta}_1 + t\hat{\theta}_2$. Denote $\Delta\hat{\theta} = \hat{\theta}_2-\hat{\theta}_1$. Note that $\hat{\theta}(t) \in [-\pi/2,\pi/2)$ and $\Delta\hat{\theta} \in (0,\pi)$. Also, using \cite[eq.~9]{us},
    \begin{align}
    \hat{r}(t) = \coth^{-1}\left(\frac{\coth \hat{r}_1 \sin((1-t)\Delta\hat{\theta}) + \coth \hat{r}_2 \sin(t\Delta\hat{\theta})}{\sin\Delta\hat{\theta}}\right). 
    \end{align}
    
    Using the inverse of $\delta_{0,(k,1)}$, we obtain $x_1, x_2, x(t)$ from $\hat{x}_1, \hat{x}_2, \hat{x}(t)$ as
    \begin{align}
    x_1 &= X(r_1,\theta_1)\nonumber\\
    &= X(\hat{r}_1(s^2\cos^2\hat{\theta}_1 + \sin^2\hat{\theta}_1)^{1/2},\tan^{-1}((1/s)\tan\hat{\theta}_1)),\\
    x_2 &= X(r_2,\theta_2)\nonumber\\
    &= X(\hat{r}_2(s^2\cos^2\hat{\theta}_2 + \sin^2\hat{\theta}_2)^{1/2},\tan^{-1}((1/s)\tan\hat{\theta}_2)),\\
    x(t) &= X(r(t), \theta(t)),
    \end{align}
    where
    \begin{align}
    r(t) &= \hat{r}(t)\sqrt{\beta(t)},\\
    \theta(t) &= \tan^{-1}((1/s)\tan\hat{\theta}(t)), \\
    \beta(t) &= s^2\cos^2\hat{\theta}(t) + \sin^2\hat{\theta}(t). 
    \end{align}
    Note that $\theta(t) \in [\theta_1,\theta_2] \subseteq [-\pi/2,\pi/2)$ and $\theta(0)=\theta_1$, $\theta(1)=\theta_2$, $r(0)=r_1$ and $r(1)=r_2$. So, $x(t)$ joins $x_1$ to $x_2$ and lies in the region between the rays from $0$ towards $x_1$ and $x_2$.
    
    \medskip

    Let $x_0(t)$ be the curve representing $[x_1,x_2]$ where $t \in [0,1]$. Then
    \begin{align}
        x_0(t)&=X(r_0(t),\theta_0(t))
    \end{align}
    where $\theta_0(t)=(1-t)\theta_1 + t\theta_2$ and $\Delta\theta=\theta_2-\theta_1$. Note that $\theta_0(t) \in [-\pi/2,\pi/2)$ and $\Delta\theta \in (0,\pi)$. Also, using \cite[eq.~9]{us},
    \begin{align}
    r_0(t) = \coth^{-1}\left(\frac{\coth r_1 \sin(t\Delta\theta) + \coth r_2 \sin((1-t)\Delta\theta)}{\sin\Delta\theta}\right). 
    \end{align}

    \medskip

    Now we show that $x(t)$ lies on the side of $[x_1,x_2]$ where $0$ lies. First, we show that the geodesic curvature $\kappa_g(t)$ of $x(t)$ is negative for all $t \in [0,1]$. Let $v(t)$ be the velocity of the curve $x(t)$. We will require the following to compute $k_g(t)$. We suppress the arguments $t$ for brevity and use primes to denote derivatives with respect to $t$. Clearly, $\hat{\theta}' = \Delta\hat{\theta}$, $\beta - s^2 = (1-s^2)\sin^2\hat{\theta} > 0$ and $1-\beta = (1-s^2)\cos^2\hat{\theta} > 0$. Also, we have
    \begin{align}
        \hat{r}' &= \frac{\Delta\hat{\theta} (\sinh^2 \hat{r})(\coth \hat{r}_1 \cos((1-t)\Delta\hat{\theta}) -\coth \hat{r}_2 \cos(t\Delta\hat{\theta}))}{\sin\Delta\hat{\theta}},\\
        \hat{r}'' &= 2\hat{r}'^2\coth \hat{r} + \Delta\hat{\theta}^2 \sinh(2\hat{r})/2, \\
        r' &= \hat{r}'\sqrt{\beta} + \hat{r}\beta'/2\sqrt{\beta}, \\
        r'' &= \hat{r}''\sqrt{\beta} + \frac{\hat{r}'\beta'}{\sqrt{\beta}} + \frac{\hat{r}}{2}\left(\frac{\beta''\sqrt{\beta}-\beta'^2/2\sqrt{\beta}}{\beta}\right), \\
        \beta' &= (1-s^2)\Delta\hat{\theta} \sin 2\hat{\theta}, \\
        \beta'^2 &= 4(\beta-s^2)(1-\beta)(\Delta\hat{\theta})^2, \\
        \beta'' &= 2(1-s^2)(\Delta\hat{\theta})^2\cos 2\hat{\theta}, \\
        \theta' &= s\Delta\hat{\theta}/\beta, \\
        \theta'' &= -\theta'\beta'/\beta. 
    \end{align}
    \begin{lemma}
    \label{k_g_x}
        Substituting the above equations in Eq. (\ref{k_g_formula}) we obtain 
        \begin{align}
            k_g &= P_0(P_1\hat{r}'^2 + P_2\hat{r}'\Delta\hat{\theta} + P_3(\Delta\hat{\theta})^2) \label{k_g}
        \end{align}
        where
        \begin{align}
            P_0 &= \frac{1}{v^3}\frac{s\Delta\hat{\theta}}{\beta}\sinh r,\\
            P_1 &= 2\sqrt{\beta}\ \frac{\psi(\hat{r}\sqrt{\beta})-\psi(\hat{r})}{\hat{r}},\\
            P_2^2 &= 16(\beta-s^2)(1-\beta)\frac{[\psi(\hat{r}\sqrt{\beta})]^2}{\beta},\\
            P_3 &= \frac{1}{2\beta\sqrt{\beta}}\left(\frac{s^2}{\sqrt{\beta}}\phi(2\hat{r}\sqrt{\beta}) - \beta^2\phi(2\hat{r}) +4\hat{r}(\beta-s^2)(1-\beta)\psi(\hat{r}\sqrt{\beta})\right).
        \end{align}
    \end{lemma}
     The proof is in the appendix. 
     
     \medskip
    
    We now show that for all $t \in [0,1]$, $P_0 > 0$ and $P_1\hat{r}'^2 + P_2\hat{r}\Delta\theta + P_3(\Delta\hat{\theta})^2 < 0$. The latter is done by showing that $P_1 < 0$ and the discriminant $P_2^2-4P_1P_3<0$. It follows that $k_g(t) < 0$ for all $t \in [0,1]$ as claimed.

    \medskip

    Clearly, $P_0 > 0$ for all $t \in [0,1]$. Since $\psi(a)$ is increasing in $a$ when $a > 0$ we have
    \begin{align}
        \psi(\hat{r}\sqrt{\beta})-\psi(\hat{r}) < 0 \label{eq1}
    \end{align}
    and therefore $P_1 < 0$.
    
    \medskip

    Finally, we show that $P_2^2-4P_1P_3 < 0$. Note that
    \begin{align}
        P_2^2-4P_1P_3 &= \frac{4}{\beta}\left\{\frac{\psi(\hat{r})-\psi(\hat{r}\sqrt{\beta})}{\hat{r}}\left(\frac{s^2}{\sqrt{\beta}}\phi(2\hat{r}\sqrt{\beta})-\beta^2\phi(2\hat{r})\right)\right.\nonumber \\
        &\ \ \qquad\qquad \left.+4(\beta-s^2)(1-\beta)\psi(\hat{r}\sqrt{\beta})\psi(\hat{r})\right\}. \label{eq2}
    \end{align}
    We need the following lemmas which are proved in the appendix.
    \begin{lemma}
    \label{lem1}
        For all $x > 0$ and $y \in (0,1)$
        {
        \begin{align}
            \sinh xy -xy < y^3(\sinh x -x). 
        \end{align}
        }%
    \end{lemma}
    Using $x=2\hat{r}$ and $y = \sqrt{\beta}$ in this lemma, we obtain
    \begin{align}
        \phi(2\hat{r}\sqrt{\beta}) < \beta\sqrt{\beta}\phi(2\hat{r}). \label{eq3}
    \end{align}
    \begin{lemma}
    \label{lem4}
        For all $x > 0$ and $y \in (0,1)$
        {
        \begin{align}
            \frac{(xy\coth xy -1)(x\coth x-1)}{\coth x - y\coth xy} < \frac{y^2}{4(1-y^2)}(\sinh 2x -2x). 
        \end{align}
        }%
    \end{lemma}
    Using $x = \hat{r}$ and $y=\sqrt{\beta}$ in this lemma, we obtain
    \begin{align}
        \frac{\psi(\hat{r}\sqrt{\beta})\psi(\hat{r})}{(\psi(\hat{r}) - \psi(\hat{r}\sqrt{\beta}))/\hat{r}} < \frac{\beta}{4(1-\beta)}\phi(2\hat{r}).
    \end{align}
    Using inequality (\ref{eq1}) we rewrite this inequality as
    \begin{align}
        \psi(\hat{r}\sqrt{\beta})\psi(\hat{r}) < \frac{\beta}{4(1-\beta)}\frac{\psi(\hat{r}) - \psi(\hat{r}\sqrt{\beta})}{\hat{r}}\phi(2\hat{r}). \label{eq4}
    \end{align}
    Substituting inequality (\ref{eq3}) in (\ref{eq2}), factoring out $4(\beta-s^2)(1-\beta)$ and then using inequality (\ref{eq4}) we obtain
    \begin{align}
    &P_2^2 - 4P_1P_3 \nonumber \\
    &\ \ < \underbrace{\frac{16(\beta-s^2)(1-\beta)}{\beta}}_{>0}\underbrace{\left(\psi(\hat{r}\sqrt{\beta})\psi(\hat{r}) -\frac{\beta}{4(1-\beta)}\frac{\psi(\hat{r})-\psi(\hat{r}\sqrt{\beta})}{\hat{r}}\phi(2\hat{r})\right)}_{<0} \nonumber\\
    &\ \ < 0 \label{eq5}
    \end{align}
    So $\kappa_g(t) < 0$ for all $t \in [0,1]$.

    \medskip

    Now, consider the curve $\gamma(t)$ parameterized by $r_{\gamma}(t)=(1-t) r_1 + tr_2$ and $\theta_{\gamma}(t) = (1-t) \theta_1 + t \theta_2$. Clearly $\gamma(t)$ runs from $x_1$ to $x_2$ and lies in the region bounded by the rays from $0$ towards $x_1$ and $x_2$. Using the fact that $\coth a$ is strictly convex when $a>0$ and the following lemma proved in the appendix (and in \cite{klen})
    \begin{lemma}
    \label{lem12}
        For all $x \in (0,\pi)$ and $y \in [0,1]$
        {
        \begin{align}
            \sin xy \geqslant y\sin x. 
        \end{align}
        }%
    \end{lemma}
    we obtain
    \begin{align}
        \coth r_{\gamma}(t) &= \coth((1-t) r_1 + t r_2)\\
        &< (1-t)\coth r_1 +t\coth r_2\\
        &< \frac{\coth r_1 \sin((1-t)\Delta\theta) + \coth r_2\sin(t\Delta\theta)}{\sin \Delta\theta}\\
        &=\coth r_0(t).
    \end{align}
    Since $\coth a$ is decreasing when $a > 0$ we have $r_{\gamma}(t) > r_0(t)$. Therefore, $\gamma(t)$ lies on the side of $[x_1,x_2]$ opposite to $0$. Let $v_{\gamma}(t)$ be the velocity of the curve $\gamma(t)$. Also, note that on $\gamma$, $E = 1$, $ G = \sinh^2r_{\gamma}$, $G_r = \sinh 2r_{\gamma} $, 
    \begin{align}
        r_{\gamma}' &= r_2-r_1 = \Delta r,\\
        r_{\gamma}'' &= 0,\\
        \theta_{\gamma}' &= \theta_2-\theta_1 = \Delta\theta \in (0,\pi),\\
        \theta_{\gamma}'' &= 0.
    \end{align}
    Substituting in Eq. (\ref{k_g_formula}) we obtain the geodesic curvature of $\gamma$ as
    \begin{align}
    \frac{1}{v_{\gamma}^3}\Delta\theta\sinh 2r_{\gamma} \left(2\coth r_{\gamma} \Delta r^2 + \frac{\sinh 2r_{\gamma} }{2}\Delta\theta^2\right) > 0.
    \end{align}
    
    So, $x(t)$ and $\gamma(t)$ are two curves from $x_1$ to $x_2$ where $x_1$ and $x_2$ are distinct and $x_1,x_2\neq 0$. These curves lie in the region bounded by rays originating from $0$ towards $x_1$ and $x_2$, and their geodesic curvatures do not change sign. The geodesic curvature of $\gamma(t)$ is positive and that of $x(t)$ is negative. Since $\gamma(t)$ lies on the side of $[x_1,x_2]$ opposite to $0$ therefore, using Lemma (\ref{lem0}), we conclude that $x(t)$ lies on the side of $[x_1,x_2]$ where $0$ lies. Finally, using the facts that $x_1,x_2,0 \in C$ and $C$ is h-convex, we conclude that $x(t) \in C$ for all $t \in [0,1]$ and therefore $C'$ is h-convex.
\end{proof}

\section{Conclusion and future work}
In this work we showed that the asymmetric expansion of a hyperbolic convex set in the Poincar\'e disk about a point inside it results in a hyperbolic convex set. Similar examples as in our previous work \cite{us} would show that the alternate cases of asymmetric dilation specifically when $k_1 < 1$ and $k_2 \geqslant 1$ may not preserve hyperbolic convexity. In our previous work \cite{us} we also studied symmetric contraction of convex sets in spherical geometry. Based on the data from computer experiments we conjecture that asymmetric spherical contraction of a spherical convex set about a point inside it such that the set is contained in the hemisphere centered at the point, preserves spherical convexity. We aim to prove this in our future work.

\section{Conflict of interest statement}
On behalf of all authors, the corresponding author states that there is no conflict of interest.

\begin{acknowledgements}
This is a pre-print of an article published in Journal of Geometry. The final authenticated version is available online at: \href{https://doi.org/10.1007/s00022-020-00545-4}{https://doi.org/10.1007/s00022-020-00545-4}.
\end{acknowledgements}



\appendix\normalsize
\section*{Appendix}

\textbf{Proof of Lemma (\ref{k_g_x})} Since $\theta' = s\Delta\hat{\theta}/\beta > 0$, we rewrite Eq. (\ref{k_g_formula}) as
\begin{align}
    k_g &= \frac{1}{v^3}\theta'\sqrt{EG}\left(\frac{G_r}{G}r'^2 +\frac{G_r}{2E}\theta'^2 +\frac{r'\theta''}{\theta'} -r'' \right). \label{k_g_formula2}
\end{align}

We denote the term outside the bracket as $P_0$. So,
\begin{align}
    P_0 &= \frac{1}{v^3}\theta'\sqrt{EG} = \frac{1}{v^3}\frac{s\Delta\hat{\theta}}{\beta}\sinh r.
\end{align}

Now, we write each term in the bracket separately and assign labels so as to group the terms which contain $\hat{r}'^2$, $\hat{r}'\Delta\hat{\theta}$ and $(\Delta\hat{\theta})^2$. Note that $\beta'$ contains $\Delta\hat{\theta}$ and $\beta''$ contains $(\Delta\hat{\theta})^2$.
\begin{align}
    \frac{G_r}{G}r'^2 &= 2\coth(\hat{r}\sqrt{\beta})(\underbrace{\hat{r}'^2\beta}_\text{(i)} + \underbrace{\frac{\hat{r}^2\beta'^2}{4\beta}}_\text{(iii)}+\underbrace{\hat{r}\hat{r}'\beta'}_\text{(ii)}),\\
    \frac{G_r}{2E}\theta'^2 &= \underbrace{\frac{\sinh 2\hat{r}\sqrt{\beta}}{2}\frac{s^2(\Delta\hat{\theta})^2}{\beta^2}}_\text{(iii)},\\
    \frac{r'\theta''}{\theta'} &= (\hat{r}'\sqrt{\beta}+\frac{\hat{r}\beta'}{2\sqrt{\beta}})\left(-\frac{\beta'}{\beta}\right) = -\underbrace{\frac{\hat{r}'\beta'}{\sqrt{\beta}}}_\text{(ii)} - \underbrace{\frac{\hat{r}\beta'^2}{2\beta\sqrt{\beta}}}_\text{(iii)},\\
    r'' &= \left(\underbrace{2\hat{r}'^2\coth \hat{r}}_\text{(i)} + \underbrace{(\Delta\hat{\theta})^2\frac{\sinh 2\hat{r} }{2}}_\text{(iii)}\right)\sqrt{\beta} + \underbrace{\frac{\hat{r}'\beta'}{\sqrt{\beta}}}_\text{(ii)},\nonumber\\
    & \ \ \qquad + \underbrace{\frac{\hat{r}}{2}\left(\frac{\beta''\sqrt{\beta}-\beta'^2/2\sqrt{\beta}}{\beta}\right)}_\text{(iii)}.
\end{align}

We now combine the terms having the same labels. From the terms with labels (i), (ii) and (iii) we obtain $P_1\hat{r}'^2$, $P_2\hat{r}'\Delta\hat{\theta}'$ and $P_3(\Delta\hat{\theta}')^2$ respectively. So,
\begin{align}
    P_1 &= 2\sqrt{\beta}(\sqrt{\beta}\coth\hat{r}\sqrt{\beta}-\coth\hat{r})\\
    &=2\sqrt{\beta}\ \frac{\psi(\hat{r}\sqrt{\beta})-\psi(\hat{r})}{\hat{r}},\\
    P_2^2 &= 4\left(\frac{\beta'}{\Delta\hat{\theta}}\right)^2\left(\hat{r}\coth\hat{r}\sqrt{\beta} - \frac{1}{\sqrt{\beta}}\right)^2\\
    &= 16(\beta-s^2)(1-\beta)(\hat{r}\coth\hat{r}\sqrt{\beta}-1/\sqrt{\beta})^2\\
    &= 16(\beta-s^2)(1-\beta)\frac{[\psi(\hat{r}\sqrt{\beta})]^2}{\beta},
\end{align}
\begin{align}
    P_3 &= \frac{s^2}{2\beta^2}\sinh 2\hat{r}\sqrt{\beta} - \frac{\sinh 2\hat{r}}{2}\sqrt{\beta} - \frac{\hat{r}\beta''}{2(\Delta\hat{\theta})^2\sqrt{\beta}} \nonumber\\
    & \qquad +\frac{\hat{r}\beta'^2}{2(\Delta\hat{\theta})^2\beta}(\hat{r}\coth\hat{r}\sqrt{\beta}-\frac{1}{2\sqrt{\beta}})\\
    &= \frac{1}{2\beta\sqrt{\beta}}\left(\underbrace{\frac{s^2}{\sqrt{\beta}}\sinh 2\hat{r}\sqrt{\beta}}_\text{(a)} - \underbrace{\beta^2\sinh 2\hat{r}}_\text{(b)} - \underbrace{2\hat{r}\beta(1-s^2)\cos 2\hat{\theta}}_\text{(d)}\right. \nonumber\\
    & \qquad\qquad\qquad +\left.4\hat{r}(\beta-s^2)(1-\beta)(\hat{r}\sqrt{\beta}\coth\hat{r}\sqrt{\beta}-\underbrace{1/2}_\text{(c)})\right). \label{eqP_3}
\end{align}
We rewrite term $\text{(d)}$ in Eq. (\ref{eqP_3}) as
\begin{align}
    2\hat{r}\beta(1-s^2)\cos 2\hat{\theta} &= \underbrace{2\hat{r}s^2}_\text{(a)} - \underbrace{2\hat{r}\beta^2}_\text{(b)} + \underbrace{2\hat{r}(1-\beta)(\beta-s^2)}_\text{(c)} 
\end{align}
and substitute it back in Eq. (\ref{eqP_3}) while combining the terms with the same labels to get
\begin{align}
    P_3 &= \frac{1}{2\beta\sqrt{\beta}}\left(\frac{s^2}{\sqrt{\beta}}(\sinh2\hat{r}\sqrt{\beta}-2\hat{r}\sqrt{\beta})- \beta^2(\sinh 2\hat{r}-2\hat{r}) \right.  \nonumber\\
    & \qquad\qquad\qquad + \left. 4\hat{r}(\beta-s^2)(1-\beta)(\hat{r}\sqrt{\beta}\coth\hat{r}\sqrt{\beta}-1)\right)\\
    &= \frac{1}{2\beta\sqrt{\beta}}\left(\frac{s^2}{\sqrt{\beta}}\phi(2\hat{r}\sqrt{\beta}) - \beta^2\phi(2\hat{r}) +4\hat{r}(\beta-s^2)(1-\beta)\psi(\hat{r}\sqrt{\beta})\right)
\end{align}
\qed

\textbf{Proof of Lemma (\ref{lem1})}. Using Taylor expansion about $x=0$ we get
\begin{align}
    \sinh xy- y^3\sinh x  -xy + xy^3 &= y^3\sum_{k=1}^{\infty}\frac{x^{2k+1}(y^{2k-1}-1)}{(2k+1)!}.
\end{align}
Since $x > 0$ and $y \in (0,1)$ each term in the above summation is negative so the overall expression is negative. 
\qed

\begin{lemma}
    \label{lem2}
    For all $x > 0$
    {
    \begin{align}
        x^3(\coth x +x(1-\coth^2 x)) - 6(x\coth x-1)^2 > 0. \label{eq6}
    \end{align}
    }%
\end{lemma}
\begin{proof}
    Since $x > 0$ we have $\sinh^2 x > 0$. So we multiply by $\sinh^2 x$ throughout and prove the resulting inequality. Our claim is equivalent to $I>0$, where $I$ is defined as
    {
    \begin{align}
        I &= x^3(\cosh x \sinh x + x(\sinh^2 x - \cosh^2 x))-6(x\cosh x - \sinh x)^2\\
        &\ \ =x^3(\sinh(2x)/2 - x) - 6(x\cosh x - \sinh x)^2 \\
        &\ \ = x^3\left(\sinh(2x)/2 - x\right) - 6\left(x^2\cosh^2 x + \sinh^2 x - x\sinh(2x)\right) \\
        &\ \ = x^3\left(\frac{\sinh 2x }{2} - x\right) - 6\left(x^2\frac{\cosh 2x +1}{2} + \frac{\cosh 2x -1}{2} - x\sinh 2x\right).
    \end{align}
    }%
    Using Taylor expansion about $x=0$ we get
    {
    \begin{align}
        x^3\left(\frac{\sinh 2x }{2} - x\right) &= \frac{2x^6}{3} + \frac{2x^8}{15} + \sum_{k=3}^{\infty}\frac{2^{2k}x^{2k+4}}{(2k+1)!},\\
        6x^2\frac{\cosh 2x +1}{2} &= 6x^2 + 6x^4 + 2x^6 + \frac{4x^8}{15} + 6\sum_{k=3}^{\infty}\frac{2^{2k+1}x^{2k+4}}{(2k+2)!},\\
        6\frac{\cosh 2x -1}{2} &= 6x^2 + 2x^4 + \frac{4x^6}{15} + \frac{2x^8}{105} + 6\sum_{k=3}^{\infty}\frac{2^{2k+3}x^{2k+4}}{(2k+4)!},\\
        6x\sinh 2x &= 12x^2 + 8x^4 + \frac{8x^6}{5} + \frac{16x^8}{105} + 6\sum_{k=3}^{\infty}\frac{2^{2k+3}x^{2k+4}}{(2k+3)!}.
    \end{align}
    }%
    Combine the terms and observe that all terms through order $x^8$ cancel to obtain
    \begin{align}
        I &= \sum_{k=3}^{\infty} \frac{2^{2k+2}x^{2k+4}}{(2k+4)!}\left((k+2)(2k+3)(k+1) - 3(2k+4)(2k+3) - 12 + 12(2k+4)\right)\\
        &= \sum_{k=3}^{\infty}\frac{2^{2k+2}x^{2k+4}}{(2k+4)!}(2k+3)(k-1)(k-2).
    \end{align}
    Clearly, for $k \geq 3$ all the terms in the summation are positive. So, $I > 0$ and therefore the inequality (\ref{eq6}) holds. 
\end{proof}
\textbf{Proof of Lemma (\ref{lem4})}. This is equivalent to showing that for all $x > 0$ and $y \in (0,1)$, $f(x,y) > 0$ where
{
\begin{align}
    f(x,y) = \frac{\coth x - y\coth xy}{(xy\coth xy-1)(x\coth x-1)} + \frac{4(1-y^{-2})}{\sinh 2x-2x}.
\end{align}
}%
Fix $x = a > 0$. Clearly as $y$ tends to $1$, $f(a,y) \rightarrow 0$. So, it is sufficient to show that $f(a,y)$ is decreasing for all $y \in (0,1)$.
{
\begin{align}
    \frac{df(a,y)}{dy} &=  \frac{1}{y^3}\left(\frac{8}{\sinh 2a-2a} - \frac{a^3y^3}{a^3}\frac{(\coth ay+ay(1-\coth^2 ay))}{(ay\coth ay-1)^2}\right).
\end{align}
}%
Since $ay > 0$, using Lemma (\ref{lem2}) we obtain 
{
\begin{align}
    \frac{df(a,y)}{dy} &<  \frac{1}{y^3}\left(\frac{8}{\sinh 2a-2a} - \frac{6}{a^3}\right) < 0,
\end{align}
}%
where the last inequality follows by using the Taylor expansion of $\sinh(2a)$ about $0$ and noting that
\begin{align}
    \frac{\sinh 2a-2a}{8} &= \frac{a^3}{6} + \sum_{k=2}^{\infty}\frac{2^{2k-2}a^{2k+1}}{(2k+1)!}> \frac{a^3}{6} \ \ \text{ since } a > 0.
\end{align}
\qed

\textbf{Proof of Lemma (\ref{lem12})}. As $x \rightarrow 0$, $\sin xy-y\sin x \rightarrow 0$ and its derivative with respect to $x$ is $y(\cos xy-\cos x) > 0$ when $x \in (0,\pi)$ and $y \in (0,1)$.
\qed

\end{document}